\RequirePackage{fix-cm}
\documentclass [smallextended,numbook,final]{svjour3}
\smartqed

\usepackage{mathptmx}
\usepackage{siunitx}
\usepackage{hyperref}
\usepackage{placeins}
\usepackage{tikz}
\usepackage{pgfplots}
\pgfplotsset{compat=1.16}
\usepackage{float}
\usepackage{amsmath,amsfonts,amssymb}
\usepackage{booktabs,multirow,multicol,arydshln}
\usepackage{comment}
\usepackage[numbers,square,comma]{natbib}
\usepackage{lscape}
\usepackage{adjustbox}

\usepackage{mleftright}

\usepackage[noabbrev,capitalise]{cleveref}
\usepackage{numprint}
\usepackage{empheq}
\usepackage{xfrac}
\usepackage{stmaryrd}
\newcommand{\tsum}[1]{{\textstyle \sum_{#1}}}
\usepackage[small]{titlesec}
\usepackage{xurl}

\usetikzlibrary{external}
\pgfplotsset{select coords between index/.style 2 args={
    x filter/.code={
        \ifnum\coordindex<#1\fi
        \ifnum\coordindex>#2\fi
    }
}}

\newcommand{\pheq}{\hphantom{{}={}}}

\DeclareMathOperator{\arrow}{{Arw}}
\DeclareMathOperator{\diag}{diag}
\DeclareMathOperator{\Diag}{Diag}

\DeclareMathOperator{\logdet}{logdet}
\DeclareMathOperator{\sdim}{sd}
\DeclareMathOperator{\SOS}{SOS}
\DeclareMathOperator{\WSOS}{WSOS}
\DeclareMathOperator{\SOSpsd}{SOSPSD}
\DeclareMathOperator{\WSOSpsd}{WSOSPSD}

\DeclareMathOperator{\cl}{cl}
\DeclareMathOperator{\intr}{int}
\newcommand{\iin}[1]{\llbracket #1 \rrbracket}

\newcommand{\bff}{\mathbf{f}}

\newcommand{\bfg}{\mathbf{g}}
\newcommand{\bfG}{\mathbf{G}}
\newcommand{\bfH}{\mathbf{H}}
\newcommand{\bfh}{\mathbf{h}}
\newcommand{\bfx}{\mathbf{x}}
\newcommand{\bfX}{\mathbf{X}}
\newcommand{\bfY}{\mathbf{Y}}
\newcommand{\bfy}{\mathbf{y}}
\newcommand{\bfs}{\mathbf{s}}
\newcommand{\bft}{\mathbf{t}}
\newcommand{\bfT}{\mathbf{T}}
\newcommand{\bfu}{\mathbf{u}}
\newcommand{\bfv}{\mathbf{v}}
\newcommand{\bfw}{\mathbf{w}}
\newcommand{\bfS}{\mathbf{S}}
\newcommand{\bfPi}{\tens{\Pi}}

\newcommand{\bfV}{\mathbf{V}}

\newcommand{\bfP}{\mathbf{P}}
\newcommand{\bfp}{\mathbf{p}}
\newcommand{\bfQ}{\mathbf{Q}}
\newcommand{\bfq}{\mathbf{q}}
\newcommand{\bbR}{\mathbb{R}}
\newcommand{\bbS}{\mathbb{S}}

\newcommand{\bbN}{\mathbb{N}}
\newcommand{\cO}{\mathcal{O}}
\newcommand{\K}{{K}}
\newcommand{\Ksos}{{K}_{\SOS}}
\newcommand{\Ksospsd}{{K}_{\SOSpsd}}
\newcommand{\Ksosso}{{K}_{\SOS \ell_2}}
\newcommand{\Ksoslo}{{K}_{\SOS \ell_1}}

\newcommand{\Kwsos}{{K}_{\WSOS}}
\newcommand{\Kwsospsd}{{K}_{\WSOSpsd}}
\newcommand{\Kwsosso}{{K}_{\WSOS \ell_2}}
\newcommand{\Kwsoslo}{{K}_{\WSOS \ell_1}}
\newcommand{\lift}{\tens{\Lambda}}
\newcommand{\lam}{\tens{\Lambda}_{\SOS}}
\newcommand{\lampsd}{\tens{\Lambda}_{\SOSpsd}}
\newcommand{\lamso}{\tens{\Lambda}_{\SOS \ell_2}}

\begin{document}


\title{Sum of squares generalizations for conic sets\thanks{
The authors would like to thank the anonymous reviewers for their helpful comments and suggestions.
This work has been partially funded by the National Science Foundation under grant OAC-1835443 and the Office of Naval Research under grant N00014-18-1-2079.}}
\author{Lea Kapelevich \and Chris Coey \and Juan Pablo Vielma}
\date{\today}

\institute{Lea Kapelevich \at Operations Research Center, MIT \\
\email{lkap@mit.edu}
\and
Chris Coey \at  Operations Research Center, MIT \\
\email{coey@mit.edu}
\and
Juan Pablo Vielma \at Google Research and \\Sloan School of Management, MIT \\
\email{jvielma@google.com}, \email{jvielma@mit.edu}
}

\maketitle

This preprint has not undergone all peer review and post-submission improvements or corrections. 
The Version of Record of this article is published in Mathematical Programming, and is available online at \url{https://doi.org/10.1007/s10107-022-01831-6}.
\\

\begin{abstract}
In polynomial optimization problems, nonnegativity constraints are typically handled using the \emph{sum of squares} condition.
This can be efficiently enforced using semidefinite programming formulations, or as more recently proposed by \citet{papp2019sum}, using the sum of squares cone directly in a nonsymmetric interior point algorithm.
Beyond nonnegativity, more complicated polynomial constraints (in particular, generalizations of the positive semidefinite, second order and $\ell_1$-norm cones) can also be modeled through structured sum of squares programs.
We take a different approach and propose using more specialized polynomial cones instead.
This can result in lower dimensional formulations, more efficient oracles for interior point methods, or self-concordant barriers with smaller parameters.
In most cases, these algorithmic advantages also translate to faster solve times in practice.
\end{abstract}


\section{Introduction}
\label{sec:introduction}

The \emph{sum of squares} (\emph{SOS}) condition is commonly used as a tractable restriction of polynomial nonnegativity.
While SOS programs have traditionally been formulated and solved using semidefinite programming (SDP), \citet{papp2019sum} recently demonstrated the effectiveness of a nonsymmetric interior point algorithm in solving SOS programs without SDP formulations.
%
%
%
In this note, we focus on structured SOS constraints that can be modeled using more specialized cones.
We describe and give barrier functions for three related cones useful for modeling functions of dense polynomials, which we hope will become useful modeling primitives.


The first is the cone of \emph{SOS matrices}, which was described by \citet[Section 5.7]{coey2021solving} without derivation.
We show that this cone can be computationally favorable to equally low-dimensional SOS formulations.
Characterizations of univariate SOS matrix cones in the context of optimization algorithms have previously been given by \citet[Section 6]{genin2003optimization}.
However, their use of monomial or Chebyshev bases complicates computations of oracles in an interior point algorithm \citep[Section 3.1]{papp2019sum} and prevents effective generalizations to the multivariate case.

The second is an \emph{SOS $\ell_2$-norm} (\emph{SOS-L2}) cone, which can be used to certify pointwise membership in the second order cone for a vector with polynomial components.
The third is an \emph{SOS $\ell_1$-norm} (\emph{SOS-L1}) cone, which can be used to certify pointwise membership in the epigraph set of the $\ell_1$-norm function.
Although it is straightforward to use SOS representations to approximate these sets, such formulations introduce cones of higher dimension than the constrained polynomial vector.
We believe we are first to describe how to handle these sets in an interior point algorithm without introducing auxiliary conic variables or constraints.
We suggest new barriers, with lower barrier parameters than SOS formulations allow. 




In the remainder of this section we provide background on SOS polynomials and implementation details of interior point algorithms that are required for later sections.
In \cref{sec:nonlinear} we describe the constraints we wish to model using each new cone, and suggest alternative SOS formulations for comparison. 
In \cref{sec:algebras} we outline how ideas introduced by \citet{papp2013semidefinite} can be used to characterize the cone of SOS matrices and the SOS-L2 cone.
\cref{sec:barriers} is focused on improving the parameter of the barriers for the SOS-L2 and SOS-L1 cones.
In \cref{sec:implementation} we outline implementation advantages of the new cones.
In \cref{sec:experiments} we compare various formulations using a numerical example and conclude in \cref{sec:conclusions}.

In what follows, we use $\bbS^m$, $\bbS_{+}^m$, and $\bbS_{++}^m$ to represent the symmetric, positive semidefinite and positive definite matrices respectively with side dimension $m$.
For sets, $\cl$ denotes the closure and $\intr$ denotes the interior.
$\iin{a..b}$ are the integers in the interval $[a, b]$.
$\vert A \vert$ denotes the dimension of a set $A$, and $\sdim(m) = \vert \bbS^m \vert = \sfrac{m (m+1)}{2}$.
We use $\langle \cdot, \cdot \rangle_{A}$ for the inner product on $A$.
For a linear operator $M: A \to B$, the adjoint $M^\ast: B \to A$ is the unique operator satisfying  $\langle x, M y \rangle_A = \langle y, M^\ast x \rangle_B$ for all $x \in A$ and $y \in B$.
$\mathbf{I}_m$ is the identity in $\bbR^{m \times m}$.
$\otimes_K: \bbR^{a_1 \times a_2} \times \bbR^{b_1 \times b_2} \to \bbR^{a_1 b_1 \times a_2 b_2}$ is the usual Kronecker product.
$\diag$ returns the diagonal elements of a matrix and $\Diag$ maps a vector to a matrix with the vector on the diagonal.
All vectors, matrices, and higher order tensors are written in bold font.
$s_i$ is the $i$th element of a vector $\bfs$ and $\bfs_{i \in \iin{1..N}}$ is the set $\{ \bfs_1, \ldots, \bfs_N \}$.
If $a, b, c, d$ are scalars, vectors, or matrices, then we use square brackets, e.g. $\begin{bsmallmatrix} a & b \\ c & d \end{bsmallmatrix}$, to denote concatenation into a matrix or vector, or round parentheses, e.g. $(a, b, c, d)$, for a general Cartesian product.
If $A$ is a vector space then $A^n$ is the Cartesian product of $n$ spaces $A$.

$\mathbb{R}[\bfx]_{n,d}$ is the ring of polynomials in the variables $\bfx = (x_1, \ldots, x_n)$ with maximum degree $d$.
Following the notation of \citet{papp2019sum}, we use $L = \binom{n+d}{n}$ and $U = \binom{n+2d}{n}$ to denote the dimensions of $\bbR[\bfx]_{n,d}$ and $\bbR[\bfx]_{n,2d}$ respectively, when $n$ and $d$ are given in the surrounding context.


\subsection{The SOS polynomials cone and generic interior point algorithms}
\label{sec:introduction:sospolynomials}

A polynomial $p(\bfx) \in \mathbb{R}[\bfx]_{n,2d}$ is SOS if it can be expressed in the form $p(\bfx) = \sum_{i \in \iin{1..N}} q_i(\bfx)^2$ for some $N \in \bbN$ and $q_{i \in \iin{1..N}}(\bfx) \in \mathbb{R}[\bfx]_{n,d}$.
We denote the set of SOS polynomials in $\mathbb{R}[\bfx]_{n,2d}$ by $\Ksos$, which is a proper cone in $\mathbb{R}[\bfx]_{n,2d}$ \citep{nesterov2000squared}.

We also say that $\bfs \in \Ksos$ for $\bfs \in \bbR^U$ if $\bfs$ represents a vector of coefficients of an SOS polynomial under a given basis.
We use such \emph{vectorized} definitions interchangeably with functional definitions of polynomial cones.
To construct a vectorized definition for $\Ksos$, suppose we have a fixed basis for $\bbR[\bfx]_{n,2d}$, and let $p_{i \in \iin{1..L}}(\bfx)$ be basis polynomials for $\bbR[\bfx]_{n,d}$.
Let $\lambda: \iin{1..L}^2 \to \mathbb{R}^U$ be a function such that $\lambda(i,j)$ returns the vector of coefficients of the polynomial $p_i(\bfx) p_j(\bfx)$ using the fixed basis for $\mathbb{R}[\bfx]_{n,2d}$.
Define the \emph{lifting operator} $\lift: \mathbb{R}^U \to \mathbb{S}^L$, introduced by \citet{nesterov2000squared}, as:
\begin{align}
    \lift(\mathbf{s})_{i,j} = \langle \lambda(i,j), \mathbf{s} \rangle_{\mathbb{R}^U} \quad \forall i,j \in \iin{1..L},
    \label{eq:general lambda}
\end{align}
where $\lift(\mathbf{s})_{ij}$ is a component in row $i$ and column $j$. 
Now the cones $\Ksos$ and $\Ksos^\ast$ admit the characterization \citep[Theorem 7.1]{nesterov2000squared}:
\begin{subequations}
\begin{align}
    \Ksos &= \lbrace \mathbf{s} \in \mathbb{R}^U:  \exists \mathbf{S} \in \bbS_{+}^L, \mathbf{s} = \lift^\ast(\mathbf{S}) \rbrace, \\
    \Ksos^\ast &= \lbrace \mathbf{s} \in \mathbb{R}^U: \lift(\mathbf{s}) \in \bbS_{+}^L \rbrace.
\end{align}
\label{eq:characterization}
\end{subequations}
\cref{eq:characterization} shows that the dual cone $\Ksos^{\ast}$ is an inverse linear image of the positive semidefinite (PSD) cone, and therefore has an efficiently computable \emph{logarithmically homogeneous self-concordant barrier} (LHSCB)
(see \citep[Definitions 2.3.1, 2.3.2]{nesterov1994interior}).
In particular, by linearity of $\lift$, the function $\bfs \mapsto -\logdet (\lift(\mathbf{s}))$ is an LHSCB for $\Ksos^\ast$ \citep[Proposition 5.1.1]{nesterov1994interior} with parameter $L$ (an $L$-LHSCB for short). 
This makes it possible to solve optimization problems over $\Ksos$ or $\Ksos^{\ast}$ with a \emph{generic} primal-dual interior point algorithm in polynomial time \citep{skajaa2015homogeneous}.%
\footnote{
We direct the interested reader to \citet{faybusovich2002self}, who obtained non-linear barriers for the cone of univariate polynomials generated by Chebyshev systems by computing the universal volume barrier of \citet{nesterov1994interior}, which is unrelated to the SDP representations of these polynomials.
}

In a generic primal-dual interior point algorithm, very few oracles are needed for each cone in the optimization problem.
For example, the algorithm described by \citet{coey2021solving} only requires a membership check, an initial interior point, and evaluations of derivatives of an LHSCB for each cone \emph{or} its dual.
Therefore, there is no particular advantage to favoring either $\Ksos$ or $\Ksos^{\ast}$ formulations.
Optimizing over $\Ksos$ (or $\Ksos^\ast$) directly instead of building SDP formulations is appealing because the dimension of $\Ksos$ is generally much smaller than the cone dimension in SDP formulations that are amenable to more specialized algorithms \citep{papp2019sum,coey2021solving}.
In later sections we describe efficient LHSCBs and membership checks for each cone we introduce.

The output of the lifting operator depends on the polynomial basis chosen for $\mathbb{R}[\bfx]_{n,d}$ as well as the basis for $\mathbb{R}[\bfx]_{n,2d}$. 
Following \citet{papp2019sum}, we use a set of Lagrange polynomials that are interpolant on some points $\mathbf{t}_{i \in \iin{1..U}}$ as the basis for $\mathbb{R}_{n,2d}[\bfx]$ and the multivariate Chebyshev polynomials \citep{hoffman1988generalized} as the basis in $\mathbb{R}_{n,d}[\bfx]$.
These choices give the particular lifting operator we implement, $\lam(\bfs)$:
\begin{align}
    \lam(\mathbf{s})_{i,j} = \tsum{u \in \iin{1..U}} p_i(\mathbf{t}_u)p_j(\mathbf{t}_u)s_u
    &
    \quad \forall i, j \in \iin{1..L}
    .
    \label{eq:scalarlambda}
\end{align}
Equivalently,
$\lam(\mathbf{s}) = \mathbf{P}^\top \Diag(\mathbf{s}) \mathbf{P}$,
where
$P_{u, \ell} = p_\ell(\mathbf{t}_u)$ for all $u \in \iin{1..U}, \ell \in \iin{1..L}$.
The adjoint $\lam^\ast: \mathbb{S}^L \to  \mathbb{R}^U$ is given by $\lam^\ast(\bfS) = \diag(\mathbf{P} \bfS \mathbf{P}^\top)$.
\citet{papp2019sum} show that the Lagrange basis gives rise to expressions for the gradient and Hessian of the barrier for $\Ksos^\ast$ that are computable in $\mathcal{O}(LU^2)$ time for any $d, n \geq 1$.
Although we assume for simplicity that $p$ is a dense basis for $\bbR[\bfx]_{n,d}$, this is without loss of generality.
A modeler with access to a suitable sparse basis of $\bar{L} < L$ polynomials in $\bbR[\bfx]_{n,d}$ and $\bar{U} < U$ interpolation points, could use \cref{{eq:scalarlambda}} and obtain a barrier with parameter $\bar{L}$.

\section{Polynomial generalizations for three conic sets}
\label{sec:nonlinear}

The first set we consider are the polynomial matrices $\bfQ(\bfx) \in \mathbb{R}[\bfx]_{n,2d}^{m \times m}$ (i.e. $m \times m$ matrices with components that are polynomials in $n$ variables of maximum degree $2d$)%
\footnote{
We assume that polynomial components in vectors and matrices involve the same variables and have the same maximum degree, to avoid detracting from the key ideas in this paper.
This assumption could be removed at the expense of more cumbersome notation.
} satisfying the constraint:
\begin{align}
    \bfQ(\bfx) \succeq 0 \quad \forall \bfx.
    \label{eq:polypsd}
\end{align}
One of the first applications of matrix SOS constraints was by \citet{henrion2006convergent}. 
The moment-SOS hierarchy was extended from the scalar case to the matrix case, using a suitable extension of Putinar's Positivstellesatz studied by \citet{hol2004sum} and \citet{kojima2003sums}.
This constraint has various applications in statistics, control, and engineering  \citep[][]{aylward2007explicit,aylward2008stability,doherty2004complete,hall2019engineering}. 
A tractable restriction for \cref{eq:polypsd} is given by the SOS formulation:
 \begin{align}
    \bfy^\top \bfQ(\bfx) \bfy \in \Ksos \quad \forall {\mathbf{y}} \in \mathbb{R}^m.
    \label{eq:scalarsospsd}
\end{align}
This formulation is sometimes implemented in practice (e.g. \citep{legat2017sum}) and requires an SOS cone of dimension $U\sdim(m)$ (by exploiting the fact that all terms are bilinear in the $\bfy$ variables). 
It is well known that \cref{eq:scalarsospsd} is equivalent to restricting  $\bfQ(\bfx)$ to be an \emph{SOS matrix} of the form  $\bfQ(\bfx) = \mathbf{M}(\bfx)^\top \mathbf{M}(\bfx)$ for some $N \in \bbN$ and $\mathbf{M}(\bfx) \in \bbR[\bfx]_{n,d}^{N \times m}$ \citep[Definition 3.76]{blekherman2012semidefinite}. To be consistent in terminology with the other cones we introduce, we refer to SOS matrices as \emph{SOS-PSD} matrices, or belonging to $\Ksospsd$.
We show how to characterize $\Ksospsd$ and use it directly in an interior point algorithm in \cref{sec:algebras}.

The second set we consider are the polynomial vectors $\mathbf{q(x)} \in \mathbb{R}[\bfx]_{n,2d}^m$ satisfying:
\begin{align}
{q}_1(\bfx) \geq \sqrt{\tsum{i \in \iin{2..m}} ( q_i(\bfx) )^2 } \quad \forall \bfx,
\label{eq:polysoc}
\end{align}
and hence requiring $\mathbf{q(x)}$ to be in the epigraph set of the $\ell_2$-norm function (second order cone) pointwise (cf. \cref{eq:polypsd} requiring the polynomial matrix to be in the PSD cone). 
A tractable restriction for this constraint is given by the SOS formulation:
\begin{align}
    \label{eq:arrowpsd}
    \bfy^\top \arrow(\mathbf{q(\bfx)}) \bfy \in \Ksos \quad \forall {\mathbf{y}} \in \mathbb{R}^m,
\end{align}
where $\arrow: \bbR[\bfx]_{n,2d}^m \to \bbR[\bfx]_{n,2d}^{m \times m}$ is defined by:
\begin{align}
\begin{split}
    \arrow(\bfp(\bfx)) = 
    \begin{bmatrix}
    p_1(\bfx) & \bar{\bfp}(\bfx)^\top \\
    \bar{\bfp}(\bfx) & p_1(\bfx) \mathbf{I}_{m-1}
    \end{bmatrix}, 
    \\
    \bfp(\bfx) = 
    ( p_1(\bfx), \bar{\bfp}(\bfx) )
    \in \bbR[\bfx]_{n,2d} \times \bbR[\bfx]_{n,2d}^{m-1}
    .
\end{split}
\end{align}
Due to the equivalence between \cref{eq:scalarsospsd} and membership in $\Ksospsd$, \cref{eq:arrowpsd} is equivalent to requiring that $\bfq(\bfx)$ belongs to the cone we denote $\K_{\arrow \SOSpsd}$ defined by:
\begin{align}
    \label{eq:arrowpsdpsd}
    \K_{\arrow \SOSpsd} =
    \{
    \bfq(\bfx) \in \bbR[\bfx]_{n,2d}^m:
    \arrow({\bfq(\bfx)}) \in \Ksospsd
    \}
    .
\end{align}
Membership in $\K_{\arrow \SOSpsd}$ ensures \cref{eq:polysoc} holds due to the SDP representation of the second order cone \citep{alizadeh2003second}, and the fact that the SOS-PSD condition certifies pointwise positive semidefiniteness. 
An alternative restriction of \cref{eq:polysoc} is described by the set we denote $\Ksosso$, which is not representable by the usual scalar polynomial SOS cone in general:
\begin{align}
    \Ksosso = \left\{ 
    \begin{aligned}
    \bfq(\bfx) \in \bbR[\bfx]_{n,2d}^m : 
    \exists N \in \bbN, \bfp_{i \in \iin{1..N}}(\bfx) \in \bbR[\bfx]_{n,d}^m,
    \\
    \bfq(\bfx) = \tsum{i \in \iin{1..N}} \bfp_i(\bfx) \circ \bfp_i(\bfx)
    \end{aligned}
    \right\},
    \label{eq:functionalksosso}
\end{align}
where $\circ: \bbR^m \times \bbR^m \to \bbR^m$ is defined by:
\begin{align}
    \bfx \circ \mathbf{y} =
    \begin{bmatrix}
    \bfx^\top \mathbf{y} \\
    x_1 \bar{\mathbf{y}} + y_1 \bar{\bfx}
    \end{bmatrix}
    , \quad \mathbf{x} = 
    (x_1, \bar{\bfx}) , \quad 
    \mathbf{y} = 
    (y_1, \bar{\bfy})
    \in \bbR \times \bbR^{m-1}
    ,
    \label{eq:circ}
\end{align}
and $\circ: \bbR[\bfx]^m_{n,d} \times \bbR[\bfx]^m_{n,d} \to  \bbR[\bfx]^m_{n,2d}$ on polynomial vectors is defined analogously.
This set was also studied by Kojima and Muramatsu with a focus on extending Positivstellensatz results \citep{kojima2007extension}.
The validity of  $\Ksosso$ as a restriction of \cref{eq:polysoc} follows from the  the characterization of the second order cone as a \emph{cone of squares} \citep[Section 4]{alizadeh2003second}. 
For this reason we will refer to the elements of $\Ksosso$ as the \emph{SOS-L$2$} polynomials.
For a polynomial vector in $\bbR[\bfx]_{n,2d}^m$, the dimension of $\Ksosso$ is $U m$, which is favorable to the dimension $U \sdim(m)$ of $\Ksos$ required for \cref{eq:arrowpsd} or $\Ksospsd$ in \cref{eq:arrowpsdpsd}. 
In addition, we show in \cref{sec:barriers:so} that $\Ksosso$ admits an LHSCB with smaller parameter than $\K_{\arrow \SOSpsd}$. 
However, we conjecture that for general $n$ and $d$, $\Ksosso \subsetneq \K_{\arrow \SOSpsd}$ (for example, consider the vector $[1 + x^2, 1 - x^2, 2 x]$, which belongs to $\K_{\arrow \SOSpsd}$ but not $\Ksosso$).
Our experiments in \cref{sec:experiments} also include instances where using $\Ksosso$ and $\K_{\arrow \SOSpsd}$ gives different objective values.
A third formulation can be obtained by modifying the SDP formulation for $\K_{\arrow \SOSpsd}$ to account for all sparsity in the $\bfy$ monomials (by introducing a specialized cone for the Gram matrix of $\bfy^\top \arrow(\bfq(\bfx)) \bfy$).
However, this approach suffers from requiring $\cO(L^2)$ conic variables for each polynomial in $\bfq(\bfx)$, so we choose to focus on $\Ksos$ and $\Ksospsd$ formulations for $\K_{\arrow \SOSpsd}$ instead.

The third and final set we consider is also described through a constraint on a polynomial vector $\bfq(\bfx) \in \mathbb{R}[\bfx]_{n,2d}^m$. 
This constraint is given by:
\begin{equation}
{q}_1(\bfx) \geq \tsum{i \in \iin{2..m}} \vert {q}_i(\bfx) \vert\quad \forall \bfx, \label{eq:polyl1}
\end{equation}
and hence requires the polynomial vector to be in the epigraph set of the $\ell_1$-norm function (\emph{$\ell_1$-norm cone}) pointwise. 
A tractable restriction for this constraint is given by the SOS formulation: 
\begin{subequations}
\begin{align}
    q_1(\bfx) - \tsum{i \in \iin{2..m}} ( p_{i}(\bfx)^{+} + p_{i}(\bfx)^{-} )
    & \in \Ksos,
    \\
    q_i(\bfx) &= p_{i}(\bfx)^{+} - p_{i}(\bfx)^{-} & \forall i \in \iin{2..m},
    \\
    p_{i}(\bfx)^{+}, p_{i}(\bfx)^{-} & \in \Ksos & \forall i \in \iin{2..m},
\end{align}
\label{eq:extl1}
\end{subequations}
which uses auxiliary polynomial variables $p_{i \in \iin{2..m}}^+(\bfx) \in \bbR[\bfx]_{n,2d}$ and $p_{i \in \iin{2..m}}^-(\bfx) \in \bbR[\bfx]_{n,2d}$. 
We refer to the projection of \cref{eq:extl1} onto $\mathbf{q(x)} \in \mathbb{R}[\bfx]_{n,2d}^m$ as $\Ksoslo$ and to its elements as the \emph{SOS-L$1$} polynomials. 
Note that the dimension of $\Ksoslo$ is $U m$, while \cref{eq:extl1} requires $2m-1$ SOS cones of dimension $U$ and $U(m - 1)$ additional equality constraints.
In \cref{sec:barriers:l1} we derive an $L m$-LHSCB that allows us to optimize over $\Ksoslo$ directly, while \cref{eq:extl1} would require an LHSCB with parameter $L (2 m  - 1)$. 

We summarize some key properties of the new cones and SOS formulations in \cref{tab:sets}: 
the total dimension of cones involved, 
the parameter of an LHSCB for the conic sets,
the time complexity to calculate the Hessian of the LHSCB (discussed in \cref{sec:implementation}), 
the level of conservatism of each new conic set compared to its alternative SOS formulation, 
and the number of auxiliary equality constraints and variables that need to be added in an optimization problem.

\begin{table}[!htb]
    \centering
    \begin{tabular}{{l}*{6}{l}}
    \toprule
        & \multicolumn{2}{l}{SOS-PSD} & \multicolumn{2}{l}{SOS-L2} & \multicolumn{2}{l}{SOS-L1} \\
        \cmidrule(lr){2-3}\cmidrule(lr){4-5}\cmidrule(lr){6-7}
         & $\Ksospsd$ & \eqref{eq:scalarsospsd} & $\Ksosso$ & \eqref{eq:arrowpsd} & $\Ksoslo$ & \eqref{eq:extl1} \\
        \midrule
        cone dim. & $U \sdim(m)$ & $U \sdim(m)$ & $U m$ & $U \sdim(m)$ &  $U m$ & $U (2m - 1)$ \\
        parameter & $L m$ & $L m$ & $2 L$ & $L m$ & $L m$ & $L (2m - 1)$ \\
        Hessian flops & $\cO(L U^2 m^3)$ & $\cO(L U^2 m^5)$ & $\cO(L U^2 m^2)$ & $\cO(L U^2 m^5)$ & $\cO(L U^2 m)$ & $\cO(L U^2 m)$ \\
        conservatism & equal & - & greater & - & equal & - \\
        equalities & 0 & 0 & 0 & 0 & 0 & $U (m-1)$ \\
        variables & 0 & 0 & 0 & 0 & 0 & $2 U (m-1)$ \\
        \bottomrule
    \end{tabular}
    \caption{Properties of new cones compared to SOS formulations. 
    }
    \label{tab:sets}
\end{table}

\section{SOS-PSD and SOS-L2 cones from general algebras}
\label{sec:algebras}

The ideas introduced by \citet{papp2013semidefinite} relating to SOS cones in \emph{general algebras} allow us to characterize $\Ksospsd$ and $\Ksosso$ without auxiliary SOS polynomial constraints.
As in \citet{papp2013semidefinite}, let us define $(A,B,\diamond)$ as a {general algebra} if $A, B$ are vector spaces and $\diamond: A \times A \to B$ is a bilinear product that satisfies the distributive property. 
For a general algebra $(A,B,\diamond)$, \citet{papp2013semidefinite} define the SOS cone $\K_\diamond$:
\begin{align}
    \K_\diamond = \{
    b \in B: \exists N \in \bbN, a_{i \in \iin{1..N}} \in A, b = \tsum{i \in \iin{1..N}} a_i \diamond a_i
    \}.
\end{align}
For instance, $\bbS_+$ is equal to the SOS cone of $(\bbR^m, \bbS^m, \overline{\diamond})$ for $\overline{\diamond}$ given by $\mathbf{x} \overline{\diamond} \mathbf{y} = \tfrac{1}{2} (\mathbf{x} \mathbf{y}^\top + \mathbf{y} \mathbf{x}^\top)$.
The second order cone is equal to the SOS cone of $(\bbR^m, \bbR^m, \circ)$.
$\Ksos$ is equal to the SOS cone of $(\mathbb{R}[\bfx]_{n,d}, \mathbb{R}[\bfx]_{n,2d}, \cdot)$ where $\cdot$ is the product of polynomials. 
To obtain our vectorized representation of $\Ksos$ we can redefine the function $\lambda: \bbR^L \times \bbR^L \to \bbR^U$ so that for $\bfp_i, \bfp_j \in \bbR^L$ representing coefficients of any polynomials in $\bbR[\bfx]_{n,d}$, $\lambda(\bfp_i, \bfp_j)$ returns the vector of coefficients of the product of the polynomials. 
Then $\Ksos$ is equal to the SOS cone of $(\bbR^L, \bbR^U, \lambda)$. 

As we describe in \cref{sec:algebras:liftingoperators}, \citet{papp2013semidefinite} also show how to build lifting operators for general algebras. 
This allows us to construct membership checks and easily computable LHSCBs for $\Ksospsd^\ast$ and $\Ksosso^\ast$ once we represent them as SOS cones of \emph{tensor products} of algebras. 

The tensor product of two algebras $(A_1, B_1, \diamond_1)$ and $(A_2, B_2, \diamond_2)$ is a new algebra $(A_1 \otimes A_2, B_1 \otimes B_2, \diamond_1 \otimes \diamond_2)$, where $\diamond_1 \otimes \diamond_2$ is defined via its action on \emph{elementary tensors}.
For $\bfu_1, \bfv_1 \in A_1$ and $\bfu_2, \bfv_2 \in A_2$:
\begin{align}
    (\bfu_1 \otimes \bfu_2) \diamond_1 \otimes \diamond_2 (\bfv_1 \otimes \bfv_2) = (\bfu_1 \diamond_1 \bfv_1) \otimes (\bfu_2 \diamond_2 \bfv_2).
\end{align}
The algebra we are interested in for a functional representation of $\Ksospsd$ is the tensor product of $(\bbR[\bfx]_{n,d}, \bbR[\bfx]_{n,2d}, \cdot)$ with $(\bbR^m, \bbS^m, \bar{\diamond})$.
We can think of elements in $\bbR[\bfx]_{n,d} \otimes \bbR^m$ as polynomial vectors in $\bbR[\bfx]_{n,d}^m$, and $\bbR[\bfx]_{n,2d} \otimes \bbS^m$ as the symmetric polynomial matrices in $\bbR[\bfx]_{n,2d}^{m \times m}$.
The SOS cone of $(\bbR[\bfx]_{n,d} \otimes \bbR^m, \bbR[\bfx]_{n,2d} \otimes \bbS^m, \cdot \otimes \bar{\diamond})$ corresponds to the polynomial matrices that can be written as $\tsum{i \in \iin{1..N}} \mathbf{m}_i(\bfx) \mathbf{m}_i(\bfx)^\top$ with $\mathbf{m}_i(\bfx) \in \bbR[\bfx]^m$ for all $i \in \iin{1..N}$ \citep[Section 4.3]{papp2013semidefinite}, which is exactly $\Ksospsd$.
Equivalently, a vectorized representation of $\Ksospsd$ can be characterized as the SOS cone of $(\bbR^L \otimes \bbR^m, \bbR^U \otimes \bbS^m, \lambda \otimes \bar{\diamond})$.
We can think of $\bbR^L \otimes \bbR^m$ as $\bbR^{L \times m}$ and we can think of $\bbR^U \otimes \bbS^m$ as a subspace of $\bbR^{U \times m \times m}$ that represents the coefficients of symmetric polynomial matrices.

Likewise, the algebra we are interested in for a functional representation of $\Ksosso$ is the tensor product of $(\bbR[\bfx]_{n,d}, \bbR[\bfx]_{n,2d}, \cdot)$ with $(\bbR^m, \bbR^m, \circ)$.
We can think of $\bbR[\bfx]_{n,d} \otimes \bbR^m$ and $\bbR[\bfx]_{n,2d} \otimes \bbR^m$ as $\bbR[\bfx]_{n,d}^m$ and $\bbR[\bfx]_{n,2d}^m$ respectively.
The SOS cone of the tensor product of these algebras then corresponds to $\Ksosso$ due to \cref{eq:functionalksosso}.
A vectorized representation of $\Ksosso$ may be characterized as the SOS cone of $(\bbR^L \otimes \bbR^m, \bbR^U \otimes \bbR^m, \lambda \otimes \circ)$.
We can think of $\bbR^U \otimes \bbR^m$ as the coefficients of polynomial vectors, represented in $\bbR^{U \times m}$.

\subsection{Lifting operators for SOS-PSD and SOS-L2}
\label{sec:algebras:liftingoperators}

The lifting operator of $(A, B, \diamond)$, when $A$ and $B$ are finite dimensional, is defined by \citet{papp2013semidefinite} as the function $\lift_\diamond: B \to \bbS^{\vert A \vert}$  satisfying $\langle a_1, \lift_\diamond(b) a_2 \rangle_A = \langle b, a_1 \diamond a_2 \rangle_B$ for all $a_1, a_2 \in A$, $b \in B$.
This leads to the following descriptions of $\K_\diamond$ and $\K_\diamond^\ast$ \citep[Theorem 3.2]{papp2013semidefinite}:
\begin{subequations}
\begin{align}
    \K_\diamond &= \{ \bfs \in B: \exists \bfS \succeq 0, \bfs = \lift_\diamond^\ast(\bfS) \},
    \label{eq:algebras:primal}
    \\
    \K_\diamond^\ast &= \{ \bfs \in B : \lift_\diamond(\bfs) \succeq 0 \}.
    \label{eq:algebras:dual}
\end{align}
\label{eq:algebras}
\end{subequations}

Recall that in order to use either $\K_\diamond$ or $\K_\diamond^\ast$ in a generic interior point algorithm, we require efficient oracles for a membership check and derivatives of an LHSCB of $\K_\diamond$ or $\K_\diamond^\ast$.
If $\lift_\diamond(\bfs)$ is efficiently computable, \cref{eq:algebras:dual} provides a membership check for $\K_\diamond^\ast$.
Furthermore, an LHSCB for $\K_\diamond^\ast$ is given by $\bfs \mapsto -\logdet ({\lift}_\diamond(\bfs))$ with barrier parameter $\vert A \vert$ due to the linearity of ${\lift}_\diamond$ \citep[Proposition 5.1.1]{nesterov1994interior}.
The following lemma describes how to compute $\lift_\diamond(\bfs)$ for a tensor product algebra.

\begin{lemma}
\label{lemma:tensorlift}
\citep[Lemma 5.1]{papp2013semidefinite}: If $\bfw_1 \in B_1$ and $\bfw_2 \in B_2$, then:
\begin{align}
    \lift_{\diamond_1 \otimes \diamond_2}(\bfw_1 \otimes \bfw_2) = \lift_{\diamond_1}(\bfw_1) \otimes_K \lift_{\diamond_2}(\bfw_2).
\end{align}
\end{lemma}

Let us define $\otimes: \bbR^U \times \bbS^m \to \bbR^{U \times m \times m}$ such that $(\bfu \otimes \bfV)_{i,j,k} = u_i V_{j,k}$ and let us represent the coefficients of a polynomial matrix by a tensor $\bfS \in \mathbb{R}^{U \times m \times m}$.
Then we may write $\mathbf{S} = \tsum{i \in \iin{1..m}, j \in \iin{1..i}} \mathbf{S}_{i,j} \otimes \mathbf{E}_{i,j}$, where $\mathbf{E}_{i,j} \in \bbR^{m \times m}$ is a matrix of zeros and ones with $E_{i,j} = E_{j,i} = 1$ and $\mathbf{S}_{i,j} \in \bbR^U$ are the coefficients of the polynomial in row $i$ and column $j$. 
Applying \cref{lemma:tensorlift}, the lifting operator for $\Ksospsd$, $\lampsd :\bbR^{U \times m \times m} \to \bbS^{L m}$ is:
\begin{subequations}
\begin{align}
    \lampsd(\mathbf{S}) = \lift_{\lambda \otimes \bar{\diamond}}(\bfS) &= \lift_{\lambda \otimes \bar{\diamond}}(\tsum{i \in \iin{1..m}, j \in \iin{1..i}}  \mathbf{S}_{i,j} \otimes \mathbf{E}_{i,j})
    \\
    &= \tsum{i \in \iin{1..m}, j \in \iin{1..i}} \lam(\mathbf{S}_{i,j}) \otimes_K \mathbf{E}_{i,j}
    .
\end{align}
\label{eq:lampsd}
\end{subequations}
The output is a block matrix, where each $L \times L$ submatrix in the $i$th group of rows and $j$th group of columns is $\lampsd(\mathbf{S})_{i,j} = \lam(\bfS_{i,j})$ for all $i, j \in \iin{1..m}$.
The adjoint operator $\lampsd^\ast: \mathbb{S}^{L m} \to \mathbb{R}^{U \times m \times m}$ may also be defined blockwise,
$\lampsd^{\ast}(\bfS)_{i,j} = \lam^\ast(\bfS_{i,j})$ for all $i, j \in \iin{1..m}$ where $\bfS_{i,j} \in \bbR^{L \times L}$ is the $(i,j)$th submatrix in $\bfS$.

Likewise, we use a tensor $\mathbf{s} \in \mathbb{R}^{U \times m}$ to describe the coefficients of a polynomial vector, and write $\mathbf{s}_i \in \mathbb{R}^U$ to denote the vector of coefficients of the polynomial in component $i$.
Applying \cref{lemma:tensorlift} again, we obtain the (blockwise) definition of the lifting operator for $\Ksosso$, $\lamso: \mathbb{R}^{U \times m} \to \bbS^{Lm}$:
\begin{align}
\lamso(\bfs)_{i,j} =  
\begin{cases}
\lam(\bfs_1)
&
i = j
\\
\lam(\bfs_j)
&
i = 1, j \neq 1
\\
\lam(\bfs_i)
&
i \neq 1, j = 1
\\
0 & \text{otherwise}
\end{cases}
& \quad \forall i, j \in \iin{1..m},
\end{align}
where $\lamso(\bfs)_{i,j} \in \bbS^L$ is the $(i,j)$th submatrix of $\lamso(\bfs)$. 
Thus $\lamso(\bfs)$ has a block arrowhead structure.
The output of the adjoint operator $\lamso^\ast: \mathbb{S}^{L m} \to \mathbb{R}^{U \times m}$ may be defined as:
\begin{align}
    \lamso^\ast(\mathbf{S})_i =
    \begin{cases}
    \tsum{j \in \iin{1..m}} \lam^{\ast}(\bfS_{j,j}) & i = 1 \\
    \lam^{\ast}(\bfS_{1,i}) + \lam^{\ast}(\bfS_{i,1}) & i \neq 1
    \end{cases}
    & \quad \forall i \in \iin{1..m},
    \label{eq:ksossodef}
\end{align}
where $\lamso^\ast(\mathbf{S})_i \in \bbR^U$ is the $i$th slice of $\lamso^\ast(\mathbf{S})$ and $\bfS_{i,j} \in \bbR^{L \times L}$ is the $(i,j)$th block in $\bfS$ for all $i, j \in \iin{1..m}$.

\section{Efficient barriers for SOS-L2 and SOS-L1}
\label{sec:barriers}

As for $\Ksospsd^\ast$ and $\Ksosso^\ast$, we show that a barrier for $\Ksoslo^\ast$ can be obtained by composing a linear lifting operator with the $\logdet$ barrier.
This is sufficient to optimize over $\Ksospsd$, $\Ksosso$ and $\Ksoslo$ without high dimensional SDP formulations.
However, for $\Ksosso^\ast$ and $\Ksoslo^\ast$ we can derive improved barriers by composing nonlinear functions with the $\logdet$ barrier instead.
We show that these compositions are indeed LHSCBs. 

\subsection{SOS-L2}
\label{sec:barriers:so}

Recall \cref{eq:algebras:dual} suggests that checking membership in $\Ksosso^\ast$ amounts to checking positive definiteness of $\lamso(\bfs)$ with side dimension $L m$.
This membership check corresponds to a \emph{straightforward} LHSCB with parameter $L m$ given by $\bfs \mapsto  -\logdet(\lamso(\bfs))$.
We now show that by working with a Schur complement of $\lamso(\bfs)$, we obtain a membership check for $\Ksosso^\ast$ that requires factorizations of only two matrices with side dimension $L$ and implies an LHSCB with parameter $2 L$.

Let $\bfPi: \bbR^{U \times m} \to \bbS^L$ return the Schur complement:
\begin{align}
    \bfPi(\bfs) =
    {\lam}(\bfs_{1}) - \tsum{i \in \iin{2..m}} {\lam}(\bfs_{i}) {\lam}(\bfs_{1})^{-1} {\lam}(\bfs_{i}).
    \label{eq:pi}
\end{align}
By \cref{eq:algebras:dual,eq:pi}:
\begin{subequations}
\begin{align}
\Ksosso^\ast & = 
    \{ \bfs \in \bbR^{U \times m}: \lamso(\bfs) \succeq 0 \} 
    \\
    & = \cl \{ \bfs \in \bbR^{U \times m}:  \lamso(\bfs) \succ 0 \} 
    \\
    & =
    \cl \{
    \bfs \in \bbR^{U \times m}:
    {\lam}(\bfs_{1}) \succ 0, 
    \bfPi(\bfs) \succ 0
    \}
    .
    \label{eq:schur}
\end{align}
\end{subequations}
\cref{eq:schur} describes a simple membership check.
Furthermore, the function $F: \bbR^{U \times m} \to \bbR$ defined by:
\begin{subequations}
\begin{align}
    F(\bfs) &= -\logdet ( \bfPi(\bfs) ) - \logdet({\lam}(\bfs_{1})) 
    \\
    &= -\logdet(\lamso(\bfs)) + (m - 2) \logdet (\lam(\bfs_1))
    ,
\end{align}
\label{eq:sosbarr}
\end{subequations}
is a $2 L$-LHSCB barrier for $\Ksosso$.

\begin{theorem}
The function $F$ defined by \cref{eq:sosbarr} is a $2 L$-LHSCB for $\Ksosso^\ast$.
\label{thm:sobarr}
\end{theorem}

\begin{proof}
It is easy to verify that $F$ is a logarithmically homogeneous barrier, so we show it is a $2L$-self-concordant barrier for $\Ksosso^\ast$.
We first show that $\hat{F}: \bbS_{++}^L \times (\bbR^{L \times L})^{m-1} \to \bbR$ defined as $\hat{F}(\bfX_1, \ldots, \bfX_m) = -\logdet(\bfX_1 - \tsum{i \in \iin{2..m}} \bfX_i \bfX_{1}^{-1} \bfX_i^{\top}) - \logdet(\bfX_1)$, is a $2 L$-self-concordant barrier for the cone:
\begin{align}
    \K_{\ell_2}^m = \cl \bigl\{
   (\bfX_1, \ldots, \bfX_m) \in \bbS^L_{++} \times (\bbR^{L \times L})^{m - 1}:
   \bfX_1 - \tsum{i \in \iin{2..m}} \bfX_i \bfX_{1}^{-1} \bfX_i^{\top} \succ 0 
   \bigr\}.
    \label{eq:blocksoc}
\end{align}
We then argue that $F$ is a composition of $\hat{F}$ with the linear map $(\bfs_1, \ldots, \bfs_m) \mapsto (\lam(\bfs_1), \ldots, \lam(\bfs_m))$ and $\Ksosso^\ast$ is an inverse image of $\K_{\ell_2}^m$ under the same map.
Then by \citet[Proposition 5.1.1]{nesterov1994interior} $F$ is self-concordant.

Let $\Gamma = \bbS^L_{+} \times (\bbR^{L \times L})^{m-1}$ and $\bfG: \intr(\Gamma) \to \bbS^L$ be defined as:
\begin{align}
    \bfG(\bfX_1, \ldots, \bfX_m) = \bfX_1 - \tsum{i \in \iin{2..m}} \bfX_i \bfX_{1}^{-1} \bfX_i^\top.
\end{align}
Let us check that $\bfG$ is $(\bbS_{+}^L, 1)$-compatible with the domain $\Gamma$ in the sense of \cite[Definition 5.1.1]{nesterov1994interior}.
This requires that 
$\bfG$ is $C^3$-smooth on $\intr(\Gamma)$, 
$\bfG$ is concave with respect to $\bbS_{+}^L$, 
and at each point $\bfX = (\bfX_1, \ldots, \bfX_m)  \in \intr(\Gamma)$ and any direction $\bfV = (\bfV_1, \ldots, \bfV_m) \in \bbS^L \times (\bbR^{L \times L})^{m-1}$ such that $-\bfX_1 \preceq \bfV_1 \preceq \bfX_1$, the directional derivatives of $\bfG$ satisfy:
\begin{align}
    \tfrac{d^3 \bfG}{d \bfX^3}[\bfV, \bfV, \bfV] \preceq
    -3 \tfrac{d^2 \bfG}{d \bfX^2}[\bfV, \bfV] 
    . 
    \label{eq:defcompat}
\end{align}

Let $\bfV \in \bbS^L \times (\bbR^{L \times L})^{m-1}$.
It can be checked that $\tfrac{d^3 \bfG}{d \bfX^3}$ is continuous on the domain of $\bfG$ and we have the directional derivatives:
\begin{align}
    \tfrac{d^2 \bfG}{d \bfX^2}[\bfV, \bfV]
    &= -2 \tsum{i \in \iin{2..m}} (\bfX_i \bfX_{1}^{-1} \bfV_1 - \bfV_i)\bfX_{1}^{-1}(\bfX_i \bfX_{1}^{-1} \bfV_1 - \bfV_i)^\top,
    \\
    \tfrac{d^3 \bfG}{d \bfX^3}[\bfV, \bfV, \bfV]
    &= 6 \tsum{i \in \iin{2..m}} (\bfX_i \bfX_{1}^{-1} \bfV_1 - \bfV_i)\bfX_{1}^{-1} \bfV_1 \bfX_{1}^{-1}(\bfX_i \bfX_{1}^{-1} \bfV_1 - \bfV_i)^\top.
\end{align}
Since $\bfX_1 \succ 0$ in $\intr(\Gamma)$, $-\tfrac{d^2 \bfG}{d \bfX^2}[\bfV, \bfV] \succeq 0$ and so by \citet[Lemma 5.1.2]{nesterov1994interior}, $\bfG$ is concave with respect to $\bbS_{+}^L$.
It remains to show that \eqref{eq:defcompat} is satisfied.
Since the directional derivatives decouple by each index $i$ in the sum, it is sufficient to show that the inequality is satisfied for each $i \in \iin{2..m}$.
For this, it is sufficient that:
\begin{align}
    6 \bfX_{1}^{-1} \bfV_1 \bfX_{1}^{-1} \preceq -3 \times -2 \bfX_{1}^{-1},
    \label{eq:compatreq}
\end{align}
for all $-\bfX_1 \preceq \bfV_1 \preceq \bfX_1$, which follows since $\bfX_1$ is positive definite on $\intr(\Gamma)$.
Now by \cite[proposition 5.1.7]{nesterov1994interior}, $\hat{F}$ is a $2L$-LHSCB.
The same is true for $F$ by composing $\hat{F}$ with a linear map.
\qed
\end{proof}

\subsection{SOS-L1}
\label{sec:barriers:l1}

By combining \cref{eq:characterization} and \cref{eq:extl1}, the $\Ksoslo$ cone admits the semidefinite representation:
\begin{align}
    \Ksoslo = \left\{ 
    \begin{aligned}
    \bfs
    \in \bbR^{U \times m}: \exists \bfS_1, \bfS_{2,+}, \bfS_{2,-} \ldots, \bfS_{m,+}, \bfS_{m,-} \in \bbS_{+}^L, 
    \\
    \bfs_1 = \lam^{\ast}(\bfS_{1}) + \tsum{i \in \iin{2..m}} \lam^{\ast}(\bfS_{i,+} + \bfS_{i,-}),
    \\
    \bfs_i = \lam^{\ast}(\bfS_{i,+}) - \lam^{\ast}(\bfS_{i,-}) \ \forall i \in \iin{2..m}
    \end{aligned}
    \right\}.
\end{align}
Its dual cone is:
\begin{align}
    \Ksoslo^\ast = 
    \bigl\{ 
    \bfs
    \in \bbR^{U \times m}: 
    \lam(\bfs_{1} + \bfs_{i}) \succeq 0, 
    \lam(\bfs_{1} - \bfs_{i}) \succeq 0 \
    \forall i \in \iin{2..m}
    \bigr\}.
    \label{eq:l1dual}
\end{align}
\cref{eq:l1dual} suggests that checking membership in $\Ksoslo^\ast$ amounts to checking positive definiteness of $2 (m - 1)$ matrices of side dimension $L$.
This membership check corresponds to a \emph{straightforward} LHSCB with parameter $2 L (m - 1)$ that is given by $\bfs \mapsto -\tsum{i \in \iin{2..m}} \logdet(\lam(\bfs_{1} + \bfs_{i}) \lam(\bfs_{1} - \bfs_{i}))$.
We now describe a membership check for $\Ksoslo^\ast$ that requires factorizations of only $m$ matrices, and corresponds to an LHSCB with parameter $L m$.

\begin{lemma}
\label{lemma:psdequiv}
The set $\{\bfX \in \bbS_{+}^L, \bfY \in \bbS^L: -\bfX \preceq \bfY \preceq \bfX \}$ is equal to $\K_{\ell_2}^2 = \cl \{\bfX \in \bbS_{++}^L, \bfY \in \bbS^L: \bfX - \bfY \bfX^{-1} \bfY \succ 0\}$.
\end{lemma}
\begin{proof}
For inclusion in one direction:
\begin{subequations}
\begin{align}
    & \cl \{ \bfX \in \bbS_{++}^L, \bfY \in \bbS^L: \bfX - \bfY \bfX^{-1} \bfY \succ 0 \} 
    \\
    & = 
    \bigl\{ \bfX \in \bbS_{+}^L, \bfY \in \bbS^L:
    \begin{psmallmatrix} \bfX & \bfY \\ \bfY & \bfX \end{psmallmatrix} \succeq 0, \ \begin{psmallmatrix} \bfX & -\bfY \\ -\bfY & \bfX \end{psmallmatrix} \succeq 0
    \bigr\}
    \\
    & \subseteq
    \bigl\{ \bfX \in \bbS_{+}^L, \bfY \in \bbS^L: 2 \bfv^\top \bfX \bfv \pm 2 \bfv^\top \bfY \bfv \geq 0, \ \forall \bfv \in \bbR^L \bigr\}
    \\
    & = \{ \bfX \in \bbS_{+}^L, \bfY \in \bbS^L: \bfX + \bfY \succeq 0, \bfX - \bfY \succeq 0 \}
    .
\end{align}

\end{subequations}

For the other direction, suppose $-\bfX \prec \bfY \prec \bfX$.
Then $\bfX \succ 0$, $\bfY + \bfX \succ 0$, $\bfX - \bfY \succ 0$.
Note that $(\bfY + \bfX) \bfX^{-1} (\bfX - \bfY) = \bfX - \bfY \bfX^{-1} \bfY$ is symmetric.
Due to \citet[Corollary 1]{subramanian1979theorem}, this product of three matrices also has nonnegative eigenvalues.
We conclude that $-\bfX \prec \bfY \prec \bfX$ implies $\bfX \succ 0$ and $\bfX - \bfY \bfX^{-1} \bfY \succeq 0$.
Since $-\bfX \preceq \bfY \preceq \bfX = \cl \{ -\bfX \prec \bfY \prec \bfX \}$, taking closures gives the result.
\qed
\end{proof}

By \cref{lemma:psdequiv} we can write the dual cone as:
\begin{align}
    \Ksoslo^\ast = 
    \cl \left\{ 
    \begin{aligned}
    \bfs \in \bbR^{U \times m}: 
    \lam(\bfs_{1}) \succ 0,
    \\
    \lam(\bfs_{1}) - \lam(\bfs_{i}) \lam(\bfs_{1})^{-1} \lam(\bfs_{i}) \succ 0,
    \forall i \in \iin{2..m}
    \end{aligned}
    \right\}
    .
    \label{eq:soslodual}
\end{align}

\begin{theorem}
The function $F: \bbR^{U \times m} \to \bbR$ given by:
\begin{align}
\begin{split}
F(\bfs) &= -\tsum{i \in \iin{2..m}} \logdet(\lam(\bfs_1) - \lam(\bfs_i) \lam(\bfs_1)^{-1} \lam(\bfs_i)) - {}
\\
& \pheq \logdet(\lam(\bfs_1))
\end{split}
\end{align}
is an $L m$-LHSCB for $\Ksoslo^\ast$.
\end{theorem}

\begin{proof}
It is easy to verify that $F$ is a logarithmically homogeneous barrier, and we show it is an $L m$-self-concordant barrier.
As in \cref{thm:sobarr}, we define an auxiliary cone:
\begin{align}
    \K_{\ell_\infty}^m &= \{ (\bfX_1, \ldots, \bfX_m) \in \bbS_{+}^L \times (\bbR^{L \times L})^{m-1} : (\bfX_1, \bfX_i) \in \K_{\ell_2}^2 \forall i \in \iin{2..m}  \}.
\end{align}
Let $\hat{F}: \bbS_{++}^L \times (\bbR^{L \times L})^{m-1} \to \bbR$ be defined as $\hat{F}(\bfX_1, \ldots, \bfX_m) = -\tsum{i \in \iin{2..m}} \logdet(\bfX_1 - \bfX_i \bfX_1^{-1} \bfX_i^\top) - \logdet(\bfX_1)$.
We argue that $\hat{F}$ is an $L m$-self-concordant barrier for $\K_{\ell_\infty}^m$.
$F$ is a composition of $\hat{F}$ with the same linear map used in \cref{thm:sobarr} and self-concordance of $F$ then follows by the same reasoning.

Let $\Gamma = \bbS_{+}^L \times (\bbR^{L \times L})^{m-1}$ and $\bfH: \intr(\Gamma) \to (\bbS_{+}^L)^{m-1}$ be defined by:
\begin{align}
    \bfH(\bfX_1, \ldots, \bfX_m) = 
    \bigl(
    \bfX_1 - \bfX_2 \bfX_{1}^{-1} \bfX_2^\top,
    \hdots,
    \bfX_1 - \bfX_m \bfX_{1}^{-1} \bfX_m^\top
    \bigr).
\end{align}
We claim that $\bfH$ is $((\bbS_+^L)^{m-1}, 1)$-compatible with the domain $\Gamma$.
This amounts to showing that for all $i \in \iin{2..m}$, the mapping $\bfH_i: \bbS_{++}^L \times \bbR^{L \times L} \to \bbS^L$, $\bfH_i(\bfX) = \bfX_1 - \bfX_i \bfX_{1}^{-1} \bfX_i^\top$ is $(\bbS_+^L, 1)$-compatible with the domain $\bbS_{+}^L \times \bbR^{L \times L}$ (the requirements for compatibility decouple for each $i$).
The latter holds since $\bfH_i$ is equivalent to the function $\bfG$ from \cref{thm:sobarr} with $m=2$. 
Then by \citet[Lemma 5.1.7]{nesterov1994interior}, $\hat{F}$ is an $L m$-self-concordant barrier. 
\qed
\end{proof}
Note that we rely on an analogy of a representation for the $\ell_\infty$-norm cone (see \citep[Section 5.1]{coey2021solving}) in \cref{eq:soslodual}.
From this we derive an LHSCB that is analogous to the $\ell_\infty$-norm cone LHSCB.
On the other hand, we are not aware of an efficient LHSCB for its dual, the $\ell_1$-norm cone, so we cannot use the same technique to derive an LHSCB for the dual of a polynomial analogy to the $\ell_\infty$-norm cone.

\section{Implementation details} 
\label{sec:implementation}

In \crefrange{sec:implementation:sospsd}{sec:implementation:soslo} we describe the gradients and Hessians of the LHSCBs for $\Ksospsd$, $\Ksosso$, and $\Ksoslo$, which are required as oracles in an algorithm like \citet{coey2021solving}.
We give computational complexities of the Hessian oracles for each cone.
All the oracles we describe are implemented in the open-source solver Hypatia \citep{coey2021solving}.%
\footnote{Available at \url{https://github.com/chriscoey/Hypatia.jl}.}

\subsection{SOS-PSD}
\label{sec:implementation:sospsd}

To draw comparisons between $\Ksospsd$ and its SOS representation \cref{eq:scalarsospsd}, let us outline how we modify the representation of $\Ksos$ from \cref{sec:introduction:sospolynomials} to account for sparsity in a polynomial of the form $\bfy^\top \mathbf{Q}(\bfx) \bfy$. 

Suppose we have interpolation points $\bft_{i \in \iin{1..U}}$ to represent $\Ksos$ in $\bbR[\bfx]_{n,2d}$.
Let $\underline{\bft}_{i \in \iin{1..\sdim(m)}}$ represent distinct points in $\bbR^m$, where at most two components in $\underline{\bft}_{i}$ equal one and the rest equal zero, for all $i \in \iin{1..\sdim(m)}$.
We can check that the Cartesian product of $\bft_{i \in \iin{1..U}}$ and $\underline{\bft}_{i \in \iin{1..\sdim(m)}}$, given by $\{(\bft_1, \underline{\bft}_1), \ldots, (\bft_U, \underline{\bft}_{\sdim(m)}) \}$ gives $U \sdim(m)$ unisolvent points.
The polynomial $\bfy^\top \mathbf{Q}(\bfx) \bfy$ from \cref{eq:scalarsospsd} is then characterized by its evaluation at these these points.

Now let $\underline{p}_{i \in \iin{1..m}}$ be polynomials in $\bbR[\bfy]_{m,1}$ such that $\underline{p}_i(y_1, \ldots, y_m) = y_i$ for all $i \in \iin{1..m}$.
Recall that for $\Ksos$ in $\bbR[\bfx]_{n,2d}$, $\bfP$ is defined by $P_{u, \ell} = p_\ell(\mathbf{t}_u)$ for all $u \in \iin{1..U}, \ell \in \iin{1..L}$.
The new matrix $\underline{\bfP}$ for the $U \sdim(m)$-dimensional SOS cone is given by $\underline{\bfP}  = \mathbf{Y} \otimes_K \bfP$, where $\mathbf{Y} \in \bbR^{\sdim(m) \times m}$, is a Vandermonde matrix of the polynomials $\underline{p}_{i \in \iin{1..m}}$ and points $\underline{\bft}_{i \in \iin{1..\sdim{(m)}}}$.
Finally, the lifting operator $\underline{\lam}(\bfs) = \underline{\bfP}^\top \diag(\bfs) \underline{\bfP}$ is of the same form as $\lam$.

\begin{lemma}
Computing the Hessian of the LHSCB of $\Ksospsd$ requires $\cO(L U^2 m^3)$ time while the Hessian of the LHSCB in the SOS formulation requires $\cO(L U^2 m^5)$ time if $m < L < U$.
\end{lemma}

\begin{proof}
Define $\mathbf{T}_{i,j}: \bbR^{U \times m \times m} \to \bbR^{U \times U}$ for all $i, j \in \iin{1..m}$:
\begin{align}
    \mathbf{T}_{i,j}(\bfS) = \mathbf{P} (\lampsd(\bfS)^{-1})_{i,j}\mathbf{P}^\top
    = \bigl( (\mathbf{I}_m \otimes_K \mathbf{P}) \lampsd(\bfS)^{-1} (\mathbf{I}_m \otimes_K \mathbf{P})^\top \bigr)_{i,j},
    \label{eq:defT}
\end{align}
where the indices ${i,j}$ reference a $U \times U$ submatrix.
For all $i, i', j, j' \in \iin{1..m}$, $u, u' \in \iin{1..U}$, the gradient and Hessian of the barrier are:%
\footnote{In practice we only store coefficients from the lower triangle of a polynomial matrix and account for this in the derivatives.}
\begin{align}
    \frac{d F}{d S_{i,j,u}} & = -\mathbf{T}_{i,j}(\bfS)_{u,u},
    \\
    \frac{d^2 F}{d S_{i,j,u} d S_{i', j',u'}} 
    &= \mathbf{T}_{i,j'}(\bfS)_{u,u'} \mathbf{T}_{j,i'}(\bfS)_{u,u'}
    .
    \label{eq:sdpgradhess}
\end{align}

The lifting operator $\lampsd$ can be computed blockwise in $\mathcal{O}(L^2 U m^2)$ operations, while $\underline{\lam}$ requires $\mathcal{O}(L^2 U m^4)$ operations.
To avoid computing the explicit inverse $\lampsd(\bfs)^{-1}$, we use a Cholesky factorization $\lampsd = \mathbf{L} \mathbf{L}^\top$ to form a block triangular matrix $\mathbf{V} = \mathbf{L}^{-1}(\mathbf{I}_m \otimes_K \bfP)^\top$ in $\mathcal{O}(L^2 U m^2)$ operations, while computing the larger $\underline{\bfV} = \underline{\mathbf{L}}^{-1} \underline{\bfP}^\top$ where $\underline{\lam}(\bfs) = \underline{\mathbf{L}} \underline{\mathbf{L}}^\top$ for the $\Ksos$ formulation requires $\mathcal{O}(L^2 U m^4)$ operations.
We use the product $\mathbf{V}^\top \mathbf{V}$ to build $\mathbf{T}_{i,j}$ for all $i, j \in \iin{1..m}$ in $\mathcal{O}(L U^2 m^3)$ operations, while calculating $\underline{\bfV}^\top \underline{\bfV}$ requires $\mathcal{O}(L U^2 m^5)$ operations.
Once the blocks $\bfT_{i,j}$ are built, the time complexity to compute the gradient and Hessian are the same for $\Ksospsd$ as for $\Ksos$.
\qed
\end{proof}

\subsection{SOS-L2}
\label{sec:implementation:sosso}

\begin{lemma}
The Hessian of the LHSCB of $\Ksosso$ requires $\cO(L U^2 m^2)$ time while the Hessian of the LHSCB in the SOS formulation requires $\cO(L U^2 m^5)$ time if $m < L < U$.
\end{lemma}

\begin{proof}
Let $\mathbf{T}_{i,j}: \bbR^{U \times m} \to \bbR^{U \times U}$ be defined as in \cref{eq:defT} for all $i, j \in \iin{1..m}$, but replacing $\lampsd$ with $\lamso$.
Let $\mathbf{R} = \mathbf{P} (\lam(\bfs_1))^{-1} \mathbf{P}^\top$.
For all $i, i' \in \iin{1..m}$, $u, u' \in \iin{1..U}$, the gradient and Hessian of the barrier are:
\begin{align}
    \frac{d F}{d s_{i,u}}  
    &= \begin{cases}
    -\tsum{j \in \iin{1..m}} \mathbf{T}_{j,j}(\bfs)_{u,u} + (m - 2)\mathbf{R}_{u,u} & i = 1 \\
    -2 \mathbf{T}_{i,1}(\bfs)_{u,u}  & i \neq 1,
    \end{cases}
    \\
\begin{split}
    \frac{d^2 F}{d s_{i,u} d s_{i',u'}} & =
    \begin{cases}
    \tsum{j \in \iin{1..m}, k \in \iin{1..m}} ( \mathbf{T}_{j,k}(\bfs)_{u,u'} )^{2} - (m - 2) ( \mathbf{R}_{u,u'} )^{2} & i = i' = 1 \\
    2 \tsum{j \in \iin{1..m}} \mathbf{T}_{j,1}(\bfs)_{u,u'} \mathbf{T}_{j,i'}(\bfs)_{u,u'} & i = 1, i'\neq 1
    \\
    2 \tsum{j \in \iin{1..m}}  \mathbf{T}_{1,j}(\bfs)_{u,u'}  \mathbf{T}_{i,j}(\bfs)_{u,u'} & i \neq 1, i'= 1
    \\
    2 ( \mathbf{T}_{1,1}(\bfs)_{u,u'} \mathbf{T}_{i,i'}(\bfs)_{u,u'} + \mathbf{T}_{i,1}(\bfs)_{u,u'} \mathbf{T}_{1,i'}(\bfs)_{u,u'} ) & i \neq 1, i' \neq 1
    .
    \end{cases}
    \end{split}
    \label{eq:sogradhess}
\end{align}
To compute the blocks $\mathbf{T}_{i,j}(\bfs)$ we require an inverse of the matrix $\lamso(\bfs)$. It can be verified that:
\begin{align}
    \lamso(\bfs)^{-1}
     = 
     \begin{bmatrix}
     \mathbf{0} & \\
       & \mathbf{I}_{m-1} \otimes_K \lam(\bfs_1)^{-1}
    \end{bmatrix}
    +
    \mathbf{U} {\bfPi}(\bfs)^{-1} \mathbf{U}^\top,
    \label{eq:blockarrowinv}
\end{align}
where,
\begin{align*}
    \mathbf{U}^\top &= 
    \begin{bmatrix}
    -\mathbf{I}_L &
    \lam(\bfs_2) \lam(\bfs_1)^{-1} &
    \lam(\bfs_3) \lam(\bfs_1)^{-1} &
    \ldots &
    \lam(\bfs_m) \lam(\bfs_1)^{-1}
    \end{bmatrix}.
\end{align*}
Computing $\bfT_{i,j}$ for all $i, j \in \iin{1..m}$ is the most expensive step in obtaining the Hessian and we do this in $\cO(L U^2 m^2)$ operations.
The complexity of computing the Hessian in the SOS formulation of $\K_{\arrow \SOSpsd}$ is the same as in the SOS formulation of $\Ksospsd$ since the cones have the same dimension.
\qed
\end{proof}

\subsection{SOS-L1}
\label{sec:implementation:soslo}

\begin{lemma}
The Hessians of the LHSCBs of $\Ksoslo$ and its SOS formulation require $\cO(L U^2 m)$ time if $m < L < U$.
\end{lemma}

\begin{proof}
Let $\mathbf{T}_{i,j}(\bfs): \bbR^{U \times m} \to \bbS^U$ for all $i \in \iin{2..m}, j \in \{ 1, 2 \}$ be defined by:
\begin{align}
    \mathbf{T}_{i,j}(\bfs) &= \mathbf{P} (\lamso((s_1, s_i))^{-1} )_{1, j} \mathbf{P}^\top  
    \\
    & = \bigl( (\mathbf{I}_m \otimes_K \mathbf{P}) \lamso((s_1, s_i))^{-1} (\mathbf{I}_m \otimes_K \mathbf{P})^\top \bigr)_{1,j}
    .
    \label{eq:defl1T}
\end{align}
For all $i, i' \in \iin{1..m}$, $u, u' \in \iin{1..U}$, the gradient and Hessian of the barrier are:
\begin{align}
    \frac{d F}{d s_{i,u}} 
    &= 
    \begin{cases}
    -2 \tsum{j \in \iin{2..m}} \mathbf{T}_{j,1}(\bfs)_{u,u} + (m - 2) \mathbf{R}_{u,u} & i = 1 \\
    - 2 \mathbf{T}_{i,2}(\bfs)_{u,u}  & i \neq 1,
    \end{cases}
    \\
    \frac{d^2 F}{d s_{i,u} d s_{i',u'}}
    &= 
    \begin{cases}
    2 \tsum{j \in \iin{2..m}, k \in \{1, 2 \}} ( \mathbf{T}_{j,k}(\bfs)_{u,u'} )^{2} - (m - 2) (\mathbf{R}_{u,u'})^{2} & i = i' = 1 
    \\
    4 \mathbf{T}_{i,1}(\bfs)_{u,u'} \mathbf{T}_{i,2}(\bfs)_{u,u'} & i \neq 1, i' = 1
    \\
    4 \mathbf{T}_{i',1}(\bfs)_{u,u'} \mathbf{T}_{i',2}(\bfs)_{u,u'} & i = 1, i' \neq 1
    \\
    2 \tsum{k \in \{1, 2 \}} (\mathbf{T}_{i,k}(\bfs)_{u,u'})^{2} & i = i' \neq 1 \\
    0 & \text{otherwise}.
    \end{cases}
    \label{eq:l1gradhess}
\end{align}
Calculating $\bfT_{i,j}(\bfs)$  for all $i \in \iin{2..m}, j \in \{ 1, 2 \}$ can be done in $\cO(L U^2 m)$ operations.
The Hessian of the SOS formulation requires computing $\cO(m)$ Hessians of SOS cones that require $\cO(LU^2)$ time.
We use the block arrowhead structure of the Hessian when applying its inverse similarly to \cref{eq:blockarrowinv}.
\qed
\end{proof}

\section{Numerical example}
\label{sec:experiments}

For each cone $(\Ksospsd, \Ksosso, \Ksoslo)$ we compare the computational time to solve a simple example with its SOS formulation from \cref{sec:nonlinear}. 
We use an example analogous to the polynomial envelope problem from \citep[Section 7.2]{papp2019sum}, but replace the nonnegativity constraint by a conic inequality.
Let $q_{i \in \iin{2..m}}(\bfx)$ be randomly generated polynomials in $\bbR_{n,2d_r}[\bfx]$.
We seek a polynomial that gives the tightest approximation to the $\ell_1$ or $\ell_2$ norm of $( q_2(\bfx), \ldots, q_{m}(\bfx) )$ for all $\bfx \in [-1, 1]^n$:
\begin{subequations}
\begin{align}
    \min_{q_1(\bfx) \in \bbR_{n,2d}[\bfx]} \int_{[-1, 1]^n} q_1(\bfx) d \bfx & :
    \\
    q_1(\bfx) & \geq \vert \vert (q_2(\bfx), \ldots, q_{m}(\bfx)) \vert \vert_p & \forall \bfx \in [-1, 1]^n
    ,
    \label{eq:polyepigraph:norm}
\end{align}
\label{eq:polyepigraph}
\end{subequations}
with $p \in \{1,2\}$ in \cref{eq:polyepigraph:norm}.

To restrict \cref{eq:polyepigraph:norm} over $[-1, 1]^n$, we use \emph{weighted sum of squares} (WSOS) formulations.
A polynomial $q(\bfx)$ is WSOS with respect to weights $g_{i \in \iin{1..K}}(\bfx)$ if it can be expressed in the form of $q(\bfx) = \tsum{i \in \iin{1..K}} g_i(\bfx) p_i(\bfx)$, where $p_{i \in \iin{1..K}}(\bfx)$ are SOS.
\citet[Section 6]{papp2019sum} show that the dual WSOS cone (we will write $\Kwsos^\ast$) may be represented by an intersection of $\Ksos^\ast$ cones.
We represent the dual \emph{weighted} cones $\Kwsospsd^\ast$, $\Kwsosso^\ast$ and $\K_{\WSOS \ell_1}^\ast$ analogously using intersections of $\Ksospsd^\ast$, $\Ksosso^\ast$ and $\K_{\SOS \ell_1}^\ast$ respectively.

Let $\bff_{i \in \iin{1..m}}$ denote the coefficients of $q_{i \in \iin{1..m}}(\bfx)$ and let $\mathbf{w} \in \bbR^U$ be a vector of quadrature weights on $[-1, 1]^n$.
A low dimensional representation of \cref{eq:polyepigraph} may be written as:
\begin{align}
    \min_{\bff_1 \in \bbR^U} \mathbf{w}^\top \bff_1 & :
    \quad
    (\bff_1, \ldots, \bff_{m}) \in \K,
    \label{eq:envelope}
\end{align}
where $\K$ is  $\Kwsosso$ or $\Kwsoslo$.
If $p=2$, we compare the $\Kwsosso$ formulation with two alternative formulations involving $\K_{\arrow \SOSpsd}$.
We use either $\Kwsospsd$ to model $\K_{\arrow \SOSpsd}$ as implied in \cref{eq:arrowpsdpsd}, or $\Kwsos$ as in \cref{eq:arrowpsd}.
For $p=1$, we build an SOS formulation by replacing \eqref{eq:envelope} with:
\begin{subequations}
\begin{align}
    \min_{\bff_1, \bfg_2, \ldots, \bfg_{m}, \bfh_2, \ldots, \bfh_{m} \in \bbR^U} \mathbf{w}^\top \bff_1 & :
    \\
    \bff_1 - \tsum{i \in \iin{2..m}} (\bfg_i + \bfh_i) & \in \Kwsos,
    \\
    \bff_i - \bfg_i + \bfh_i &= 0, \quad \bfg_i, \bfh_i \in \Kwsos & \forall i \in  \iin{2..m}.
\end{align}
\label{eq:envelope:l1ext}
\end{subequations}

We select interpolation points using a heuristic adapted from \citep{papp2019sum,sommariva2009computing}.
We uniformly sample $N$ interpolation points, where $N \gg U$.
We form a Vandermonde matrix of the same structure as the matrix $\bfP$ used to construct the lifting operator, but using the $N$ sampled points for rows.
We perform a QR factorization and use the first $U$ indices from the permutation vector of the factorization to select $U$ out of $N$ rows to keep.

All experiments are performed on hardware with an AMD Ryzen 9 3950X 16-Core Processor (32 threads) and 128GB of RAM, running Ubuntu 20.10, and Julia 1.8 \citep{bezanson2017julia}.
Optimization models are built using JuMP \citep{LubinDunningIJOC} and solved with Hypatia 0.5.3 \citep{coey2021solving} 
using our specialized, predefined cones.
Scripts we use to run our experiments and raw results are available in the Hypatia repository.%
\footnote{Instructions to repeat our experiments are at \url{https://github.com/chriscoey/Hypatia.jl/tree/master/benchmarks/natvsext}.}
We use default settings in Hypatia and set relative optimality and feasibility tolerances to $10^{-7}$.

In \cref{tab:envelope,tab:norml1:master}, we show Hypatia's termination status, number of iterations, and solve times for $n \in \{1,4\}$ and varying values of $d_r$ and $m$.
We use symbols to represent the termination status, which are described in \cref{sec:results}.
If $p = 1$, we let $d = d_r$, where the maximum degree of $q_1(\bfx)$ is $2 d$.
If $p = 2$, we vary $d \in \{d_r, 2 d_r \}$ and add an additional column \emph{obj} in \cref{tab:envelope} to show the ratio of the objective value under the $\Kwsos$ (or equivalently $\Kwsospsd$) formulation divided by the objective value under the $\Kwsosso$ formulation.
Note that in our setup, the dimension of $\Kwsosso$ only depends on $d$.
A more flexible implementation could allow polynomial components to have different degrees in $\Kwsosso$ for the $d = 2 d_r$ case.

For $p=2$ and $d = 2 d_r$, the difference in objective values between $\Kwsosso$ and alternative formulations is less than $1\%$ across all converged instances.
For $p=2$ and  $d = d_r$, the difference in the objective values is around $10$--$43\%$ across converged instances.
However, the solve times for $\Kwsosso$ with $d = 2 d_r$ are sometimes faster than the solve times of alternative formulations with $d = d_r$ and equal values of $n$, $m$, and $d_r$.
This suggests that it may be beneficial to use $\Kwsosso$ in place of SOS formulations, but with higher maximum degree in the $\Kwsosso$ cone.
The solve times using $\Kwsospsd$ are slightly faster than the solve times using $\Kwsos$.
For the case where $p=1$, the $\Kwsoslo$ formulation is faster than the $\Kwsos$ formulation, particularly for larger values of $m$.
We also observe that the number of iterations the algorithm takes for $\Kwsosso$ compared to alternative formulations varies, but larger for $\Kwsoslo$ compared to the alternative SOS formulation.

\section{Conclusions}
\label{sec:conclusions}

SOS generalizations of PSD, $\ell_2$-norm and $\ell_1$-norm constraints can be modeled using specialized cones that are simple to use in a generic interior point algorithm.
The characterizations of $\Ksospsd$ and $\Ksosso$ rely on ideas from \citet{papp2013semidefinite} as well as the use of a Lagrange polynomial basis for efficient oracles in the multivariate case.
For the $\Ksospsd$ barrier, the complexity of evaluating the Hessian is reduced by a factor of $\cO(m^2)$ from the SOS formulation barrier.
This does not result in significant speed improvements since Hessian evaluations are not the bottleneck in an interior point algorithm.
In contrast, the dimension and barrier parameter of the $\Ksosso$ and $\Ksoslo$ cones are lower compared to their SOS formulations, and the complexity of evaluating the Hessian of the $\Ksosso$ barrier is reduced by a factor of $\cO(m^3)$ from its SOS formulation.
For both $\Ksosso$ and $\Ksoslo$, the total solve time was generally lower compared to their SOS formulations.
While there is no penalty in using lower dimensional representations of SOS-L1 constraints, SOS-L2 formulations give rise to more conservative restrictions than higher dimensional SOS formulations, which is observable in practice.

\appendix
\section{Result tables}
\label{sec:results}

The termination status (\emph{st}) columns of \cref{tab:envelope,tab:norml1:master} use the following codes to classify solve runs: 
\begin{description}[font=\normalfont\itshape]
\item[co] the solver claims the primal-dual certificate returned is optimal given its numerical tolerances,
\item[tl] a limit of 1800 seconds is reached,
\item[rl] a limit of approximately $120$GB of RAM is reached,
\item[sp] the solver terminates due to slow progress during iterations,
\item[er] the solver reports a different numerical error,
\item[sk] we skip the instance because the solver reached a time or RAM limit on a smaller instance.
\end{description}

\begin{table}[!htb]
\sisetup{
table-text-alignment = right,
table-auto-round,
table-figures-integer = 4,
table-figures-decimal = 1,
table-format = 4.1,
add-decimal-zero = false,
add-integer-zero = false,
}
\footnotesize
\begin{tabular}{rrrrrrSrrSrrSS[table-format = 2.2]}
\toprule
 &  & & & \multicolumn{3}{l}{$\Ksosso$} & \multicolumn{3}{l}{$\Ksos$} & \multicolumn{3}{l}{$\Ksospsd$} & 
\\
\cmidrule(lr){5-7} \cmidrule(lr){8-10} \cmidrule(lr){11-13} 
$n$ & $d_r$ & $m$ & $d$ & {st} & {iter} & {time} & {st} & {iter} & {time} & {st} & {iter} & {time} & {obj} \\
\midrule
\multirow{20}{*}{1}
 & \multirow{10}{*}{20}
  & \multirow{2}{*}{4}
  & 20 & co & 13 & .05 & co & 17 & .44 & co & 13 & .21 & .89 \\
 & & & 40 & co & 16 & .21 & co & 19 & 1.8 & co & 15 & 1.1 & .99 \\
 & & \multirow{2}{*}{8}
  & 20 & co & 13 & .12 & co & 17 & 2.9 & co & 14 & 2.1 & .85 \\
 & & & 40 & co & 19 & .70 & co & 21 & 18. & co & 16 & 10. & 1.0 \\
 & & \multirow{2}{*}{16}
  & 20 & co & 14 & .39 & co & 19 & 48. & co & 14 & 27. & .80 \\
 & & & 40 & co & 21 & 2.4 & co & 20 & 264. & co & 17 & 188. & 1.0 \\
 & & \multirow{2}{*}{32}
  & 20 & co & 15 & 1.6 & co & 22 & 1189. & co & 17 & 843. & .78 \\
 & & & 40 & co & 23 & 13. & tl & 3 & 2033. & tl & 7 & 2075. & .03 \\
 & & \multirow{2}{*}{64}
  & 20 & co & 17 & 8.5 & rl & $\ast$ & $\ast$ & rl & $\ast$ & $\ast$ & $\ast$ \\
 & & & 40 & co & 20 & 59. & sk & $\ast$ & $\ast$ & sk & $\ast$ & $\ast$ & $\ast$ \\
\cmidrule(lr){2-14}
 & \multirow{10}{*}{40}
  & \multirow{2}{*}{4}
  & 40 & co & 14 & .17 & co & 17 & 1.4 & co & 14 & 1.0 & .89 \\
 & & & 80 & co & 19 & 1.0 & co & 19 & 7.7 & co & 17 & 6.2 & .99 \\
 & & \multirow{2}{*}{8}
  & 40 & co & 16 & .57 & co & 19 & 15. & co & 15 & 9.1 & .82 \\
 & & & 80 & co & 21 & 3.1 & co & 21 & 93. & co & 17 & 62. & 1.0 \\
 & & \multirow{2}{*}{16}
  & 40 & co & 17 & 2.0 & co & 20 & 246. & co & 16 & 152. & .79 \\
 & & & 80 & co & 27 & 13. & co & 21 & 1737. & co & 18 & 1206. & 1.0 \\
 & & \multirow{2}{*}{32}
  & 40 & co & 18 & 7.6 & tl & 3 & 2031. & tl & 8 & 1803. & .02 \\
 & & & 80 & co & 27 & 53. & rl & $\ast$ & $\ast$ & rl & $\ast$ & $\ast$ & $\ast$ \\
 & & \multirow{2}{*}{64}
  & 40 & co & 19 & 36. & sk & $\ast$ & $\ast$ & sk & $\ast$ & $\ast$ & $\ast$ \\
 & & & 80 & co & 26 & 226. & sk & $\ast$ & $\ast$ & sk & $\ast$ & $\ast$ & $\ast$ \\
\cmidrule(lr){1-14}
\multirow{20}{*}{4}
 & \multirow{8}{*}{2}
  & \multirow{2}{*}{4}
  & 2 & co & 13 & .15 & co & 18 & .91 & co & 15 & .59 & .75 \\
 & & & 4 & co & 21 & 33. & co & 43 & 133. & co & 37 & 97. & 1.0 \\
 & & \multirow{2}{*}{8}
  & 2 & co & 13 & .41 & co & 21 & 11. & co & 18 & 7.7 & .64 \\
 & & & 4 & co & 21 & 102. & tl & 49 & 1816. & tl & 60 & 1811. & 1.0 \\
 & & \multirow{2}{*}{16}
  & 2 & co & 15 & 2.3 & co & 30 & 242. & co & 25 & 203. & .59 \\
 & & & 4 & co & 21 & 437. & sk & $\ast$ & $\ast$ & sk & $\ast$ & $\ast$ & $\ast$ \\
 & & \multirow{2}{*}{32}
  & 2 & co & 15 & 10. & tl & 6 & 1848. & tl & 10 & 1972. & 15. \\
 & & & 4 & co & 22 & 1707. & sk & $\ast$ & $\ast$ & sk & $\ast$ & $\ast$ & $\ast$ \\
 & & \multirow{2}{*}{64}
  & 2 & co & 15 & 46. & sk & $\ast$ & $\ast$ & sk & $\ast$ & $\ast$ & $\ast$ \\
 & & & 4 & tl & 10 & 1935. & sk & $\ast$ & $\ast$ & sk & $\ast$ & $\ast$ & $\ast$ \\
\cmidrule(lr){2-14}
 & \multirow{6}{*}{4}
  & \multirow{2}{*}{4}
  & 4 & co & 17 & 11. & co & 30 & 114. & co & 27 & 93. & .69 \\
 & & & 8 & tl & 10 & 1840. & rl & $\ast$ & $\ast$ & tl & $\ast$ & $\ast$ & $\ast$ \\
 & & \multirow{1}{*}{8}
  & 4 & co & 18 & 42. & co & 34 & 1494. & co & 29 & 1111. & .58 \\
 & & \multirow{1}{*}{16}
  & 4 & co & 18 & 174. & rl & $\ast$ & $\ast$ & tl & $\ast$ & $\ast$ & $\ast$ \\
 & & \multirow{1}{*}{32}
  & 4 & co & 16 & 580. & sk & $\ast$ & $\ast$ & sk & $\ast$ & $\ast$ & $\ast$ \\
   & & \multirow{1}{*}{64}
  & 4 & tl & 10 & 1853. & sk & $\ast$ & $\ast$ & sk & $\ast$ & $\ast$ & $\ast$ \\
\bottomrule
\end{tabular}
\caption{Solve time in seconds and number of iterations (iter) for instances with $p = 2$.}
\label{tab:envelope}
\end{table}

\begin{table}[!htb]
\sisetup{
table-text-alignment = right,
table-auto-round,
table-figures-integer = 4,
table-figures-decimal = 1,
table-format = 4.1,
add-decimal-zero = false,
add-integer-zero = false,
}
\footnotesize
\begin{tabular}{rrrrrSrrS}
\toprule
 &  & & \multicolumn{3}{l}{$\Ksoslo$} & \multicolumn{3}{l}{$\Ksos$} \\
\cmidrule(lr){4-6} \cmidrule(lr){7-9} 
$n$ & $d$ & $m$ & {st} & {iter} & {time} & {st} & {iter} & {time}\\\midrule
\multirow{10}{*}{1}
 & \multirow{5}{*}{40}
 & 8 & {co} & 17 & .53 & {co} & 15 & .48 \\
 &  & 16 & {co} & 21 & 1.3 & {co} & 15 & 1.9 \\
 &  & 32 & {co} & 25 & 3.2 & {co} & 15 & 11. \\
 &  & 64 & {co} & 29 & 7.6 & {co} & 17 & 87. \\
 &  & 128 & {co} & 32 & 17. & {co} & 18 & 610. \\
\cmidrule(lr){2-9}
 & \multirow{5}{*}{80}
 & 8 & {co} & 21 & 2.6 & {co} & 18 & 2.6 \\
 &  & 16 & {co} & 24 & 5.6 & {co} & 17 & 13. \\
 &  & 32 & {co} & 27 & 13. & {co} & 18 & 89. \\
 &  & 64 & {co} & 31 & 31. & {co} & 18 & 600. \\
 &  & 128 & {co} & 38 & 83. & tl & $\ast$ & $\ast$ \\
\cmidrule(lr){1-9}
\multirow{10}{*}{4}
 & \multirow{5}{*}{2}
 & 8 & {co} & 17 & .49 & {co} & 17 & .37 \\
 &  & 16 & {co} & 18 & 1.0 & {co} & 16 & 1.3 \\
 &  & 32 & {co} & 24 & 2.8 & {co} & 17 & 7.8 \\
 &  & 64 & {co} & 27 & 6.4 & {co} & 17 & 57. \\
 &  & 128 & {co} & 30 & 14. & {co} & 17 & 400. \\
\cmidrule(lr){2-9}
 & \multirow{5}{*}{4}
 & 8 & {co} & 25 & 28. & {co} & 21 & 54. \\
 &  & 16 & {co} & 28 & 86. & {co} & 22 & 318. \\
 &  & 32 & {co} & 29 & 198. & tl & 9 & 1823. \\
 &  & 64 & {co} & 31 & 423. & sk & $\ast$ & $\ast$ \\
&  & 128 & {co} & 42 & 1210. & sk & $\ast$ & $\ast$ \\
\bottomrule
\end{tabular}
\caption{Solve time in seconds and number of iterations (iter) for instances with $p = 1$.}
\label{tab:norml1:master}
\end{table}

\bibliography{main}

\end{document}